\documentclass[12pt]{article}
\usepackage{amsmath,amsfonts,amsthm,amssymb,enumerate,multicol,url}
\newtheorem{theorem}{Theorem}
\newtheorem{lemma}[theorem]{Lemma}
\newtheorem{corollary}[theorem]{Corollary}
\theoremstyle{definition}

\theoremstyle{remark}

\newcommand{\C}{{\mathbb{C}}}
\newcommand{\F}{{\mathbb{F}}}
\newcommand{\ff}{\mathbf{f}}

\newcommand{\R}{{\mathbb{R}}}

\begin{document}

\title{Generalizations of the Szemer\'edi-Trotter Theorem}
\author{Saarik Kalia \and Micha Sharir \and Noam Solomon \and Ben Yang}
\maketitle

\begin{abstract}
We generalize the Szemer\'edi-Trotter incidence theorem, to bound the number of complete 
\emph{flags} in higher dimensions. Specifically, for each $i=0,1,\ldots,d-1$, we are given 
a finite set $S_i$ of $i$-flats in $\R^d$ or in $\C^d$, and a (complete) flag is a tuple
$(f_0,f_1,\ldots,f_{d-1})$, where $f_i\in S_i$ for each $i$ and $f_i\subset f_{i+1}$ for each $i=0,1,\ldots,d-2$. 
Our main result is an upper bound on the number of flags which is tight in the worst case.

We also study several other kinds of incidence problems, including (i) incidences between 
points and lines in $\R^3$ such that among the lines incident to a point, at most $O(1)$
of them can be coplanar, (ii) incidences with Legendrian lines in $\R^3$, a special class of
lines that arise when considering flags that are defined in terms of other groups, and (iii)
flags in $\R^3$ (involving points, lines, and planes), where no given line can contain too 
many points or lie on too many planes. The bound that we obtain in (iii) is nearly tight in 
the worst case.

Finally, we explore a group theoretic interpretation of flags, a generalized version of
which leads us to new incidence problems.
\end{abstract}

\section{Introduction}\label{Intro}

Our starting point is the classical 1983 result of Szemer\'edi and Trotter~\cite{SzT}, which gives a worst-case
tight upper bound on the number of incidences between points and lines in the plane. Specifically, for a finite set $P$
of distinct points and a finite set $L$ of distinct lines in the plane, the number of incidences between
the points of $P$ and the lines of $L$, denoted $I(P,L)$, is the number of pairs $(p,\ell)\in P\times L$ such that $p\in\ell$.
One then has the following result.
\begin{theorem}[Szemer\'edi and Trotter~\cite{SzT}] \label{ST}
Let $P$ be a finite set of distinct points in $\R^2$ and $L$ a finite set of distinct lines in $\R^2$. Then 
$$
I(P,L) = O\left( |P|^{2/3}|L|^{2/3}+|P|+|L| \right).
$$
The bound is asymptotically tight in the worst case for any values of $|P|$ and $|L|$.
\end{theorem}\noindent
The extension of this result to the complex case is also known:
\begin{theorem}[T\'oth~\cite{T} and Zahl~\cite{Z}] \label{TZ}
Let $P$ be a finite set of distinct points in $\C^2$ and $L$ a finite set of distinct lines in $\C^2$. Then 
$$
I(P,L) = O\left( |P|^{2/3}|L|^{2/3}+|P|+|L| \right).
$$
\end{theorem}\noindent

In this paper we study the following generalization of the Szemer\'edi--Trotter bound. 
For $i=0,\ldots,d-1$, let $S_i$ be a finite set of distinct $i$-flats in $\R^d$ or in $\C^d$.
A (complete) \emph{flag} spanned by $\prod_{i=0}^{d-1} S_i$
is a sequence $(f_0,f_1,\ldots,f_{d-1})$, such that $f_i\in S_i$ for each $i$, and $f_i\subset f_{i+1}$
for $i=0,\ldots,d-2$. In this setup, Theorem~\ref{ST} bounds the number of flags spanned by $P\times L$ 
in the real plane, while Theorem~\ref{TZ} does the same for the complex plane.
Denote by $I(S_0,S_1,\ldots,S_{d-1})$ the number of flags spanned by $\prod_{i=0}^{d-1} S_i$.
Our main result is the following.

\begin{theorem}\label{Flags}
Let $S_i$ be a finite set of distinct $i$-flats in $\R^d$ or in $\C^d$, for $0\leq i\leq d-1$. Then
\begin{equation} \label{eq:flags}
I(S_0,\ldots,S_{d-1}) = O\left(\sum_{(a_0,\ldots,a_{d-1})\in\left\{0,{2/3},1\right\}^d}\prod_{i=0}^{d-1}|S_i|^{a_i} \right),
\end{equation}
where the constant of proportionality depends on $d$, and
where the ordered $d$-tuples $(a_0,\ldots,a_{d-1})$ in the sum are such that

\noindent (i)
no three consecutive $a_i$'s are nonzero,

\noindent (ii)
every $1$ is preceded and followed by $0$'s (if possible),

\noindent (iii)
every ${2/3}$ is either preceded or followed by another ${2/3}$, and

\noindent (iv)
every $0$ is either preceded or followed by a nonzero $a_i$.

\noindent
In other words, $(a_0,\ldots,a_{d-1})$ consists of pairs of consecutive $2/3$'’s or solitary $1$'’s separated by
one or two $0$'’s.  The bound is tight in the worst case.
\end{theorem}
For instance, the bound for $d=4$ is
\begin{align*}
I(S_0,S_1,S_2,S_3) & = O\left(
|S_0|^{2/3}|S_1|^{2/3}|S_3| + |S_0||S_2|^{2/3}|S_3|^{2/3}+|S_0||S_2| \right. \\
& + \left. |S_0||S_3|+|S_1||S_3|+|S_1|^{2/3}|S_2|^{2/3} \right).
\end{align*}

In Section \ref{Variants}, we consider three variants of the incidence problem in $\R^3$
(we do not know whether they extend to $\C^3$ too).  
Theorem~\ref{Flags} seems to suggest that incidence bounds in higher dimensions are just suitable variants
of the planar Szemer\'edi--Trotter bound, where the only exponents that play a role are $0$, $1$, and $2/3$.
This of course is not the case, as evidenced from the rich body of works on incidences in higher dimensions,
inaugurated by the seminal studies of Guth and Katz~\cite{GK1,GK2}. For example, as shown in \cite{GK2},
the number of incidences between a set $P$ of $m$ distinct points and a set $L$ of $n$ distinct lines in $\R^3$,
such that no plane contains more than $B$ lines of $L$, is
\begin{equation} \label{eq:gk}
I(P,L) = O\left( |P|^{1/2}|L|^{3/4} + |P|^{2/3}|L|^{1/3}B^{1/3} + |P| + |L| \right) .
\end{equation}
The variants that we will consider in Section \ref{Variants} will lead to bounds with similar exponents.
The first problem involves incidences between points and lines in $\R^3$ such that, among the lines
incident to an input point, at most $O(1)$ can be coplanar. This generalizes the special case of incidences
with equally inclined lines (also referred to as ``light-like'' lines), which are lines that form a fixed
angle with, say, the $z$-axis. This special case has been studied by Sharir and Welzl~\cite{SW}, where several
upper and lower bounds have been established. Using a simplified version of the Guth--Katz machinery, we 
rederive one of these bounds, which, resembling (\ref{eq:gk}), is 
$O\left( |P|^{1/2}|L|^{3/4} + |P| + |L| \right)$, for the more general setup
assumed above; see Theorem~\ref{Light}.
(Note that a direct application of Guth and Katz~\cite{GK2} does not seem to work, because
we do not assume that no plane contains too many (concretely, $O(|L|^{1/2})$) lines.)

The second variant involves incidences between points and \emph{Legendrian lines} in $\R^3$; these are lines
that are orthogonal to the vector field $\F(x,y,z) = (-y,x,1)$ at each of their points. We exploit the
special properties of such lines to obtain the same bound
$O\left( |P|^{1/2}|L|^{3/4} + |P| + |L| \right)$ for the number of incidences with Legendrian lines.

The third variant involves flags formed by points, lines, and planes in $\R^3$, under the assumption
that each input line contains at most $b$ input points and is contained in at most $b$ input planes, for some
parameter $b$. Using the Guth--Katz bound, we establish a sharp bound on the number of flags, which,
for $|P|=|L|=|S|=N$, becomes 
$O\left(\min \{b^2N,\;N^{3/2}\log b+bN\}\right)$; see Theorem~\ref{3D'} for the details, including
thee general form of this bound. We also show that this bound is
nearly tight in the worst case (it is tight up to the $\log b$ factor). 

The reason why the results in Section~\ref{Variants} only hold over the reals is that they use,
as in~\cite{GK2}, the polynomial partitioning technique, which holds only over $\R$.

The motivation to study incidences with Legendrian lines is that they arise in 
a Lie group generalization of the flag variety, replacing the special linear group by the 
symplectic group (of dimension 4). We discuss this generalization in the concluding section, 
where the Lie algebraic context is explained in more detail. There is a fairly extensive study of
generalized flag varieties (see, e.g., \cite{BE}). This topic falls out of the main scope of this
paper, but it serves to motivate the study of new kinds of incidence and flag problems, an initial
step of which is taken in this paper.

\section{Counting Flags in $\R^d$ and $\C^d$}\label{General}
In this section we generalize the Szemer\'edi--Trotter Theorem (and its complex counterpart)
to flags in higher dimensions, establishing Theorem~\ref{Flags}. 
The analysis holds over both the real and the complex fields. This is because the proof is recursive, and at its bottom we face
incidences (i.e., containments) between $i$-flats and $(i+1)$-flats. As the following lemma shows, this can be reduced to
incidences between points and lines in a common plane, and the bound is then obtained via Theorem~\ref{ST} or Theorem~\ref{TZ}. 

Throughout this section, $S_i$ denotes a finite set of distinct $i$-flats (in some space $\R^d$ or $\C^d$),
and we will be considering $I(S_0,\ldots,S_{d-1})$, the number of flags spanned by $S_0\times\cdots\times S_{d-1}$.
We denote the right-hand side of (\ref{eq:flags}) by $f_{d-1}(|S_0|,\ldots,|S_{d-1}|)$, 
including a suitable sufficiently large constant of proportionality.

The proof of the theorem will also involve \emph{partial flags}.
That is, let $\ff = (f_0,\ldots,f_{d-1})$ be a flag in $S_0\times \cdots\times S_{d-1}$.  
For each pair of indices $0\le i\le j\le d-1$, put $\ff_{i,j} : = (f_i,\ldots,f_j)$, 
and refer to it as an $(i,j)$-(partial) flag. As for complete flags, we denote the number
of partial $(i,j)$-flags by $I(S_i,S_{i+1},\ldots,S_j)$, where the indices indicate the
dimension of the flats in the respective sets, and the dimension of the ambient space
is implicit in (and actually irrelevant for) this notation.

\begin{lemma}\label{Lift}
Let $K$ be $\R$ or $\C$. Suppose that
$$
I(S_0,\ldots,S_{d-1}) \le f_{d-1}(|S_0|,\ldots,|S_{d-1}|),
$$
for every collection of finite sets of distinct flats $S_0,\ldots,S_{d-1}$ in $K^d$,
where each $S_i$ consists of $i$-flats, for $i=0,\ldots,d-1$. Then, for all $0\le j\le m$,
$$
I(S_{j},\ldots,S_{j+d-1}) \le f_{d-1}(|S_{j}|,\ldots,|S_{j+d-1}|),
$$
for every collection of finite sets of distinct flats $S_{j},\ldots,S_{j+d-1}$ in $K^{d+m}$,
where, again, each $S_i$ consists of $i$-flats, for $i=j,\ldots,j+d-1$. 
\end{lemma}
\begin{proof}
Suppose that the bound in the hypothesis is true and that we are given finite sets 
$S_{j},\ldots,S_{j+d-1}$ of distinct flats of the corresponding dimensions in $K^{d+m}$. 
Let $\pi$ be a generic $(d+m-j)$-flat in $K^{d+m}$, so that, in particular, 
$\pi$ is not parallel to, nor contains any of the elements of $S_{j},\ldots,S_{j+d-1}$, 
and, for each $t=0,\ldots,d-1$, the intersection of any $(j+t)$-flat in $S_{j+t}$ with 
$\pi$ is a $t$-flat. (This is possible since all the sets $S_{j+t}$ are finite.)  
Moreover, if, for some $s<t$, a $(j+s)$-flat of $S_{j+s}$ is contained in a $(j+t)$-flat 
of $S_{j+t}$, then the intersections of these flats with $\pi$ will also be contained in one another.
Therefore, putting $S'_{t} := \{f\cap \pi \mid f\in S_{j+t}\}$, for each $t=0,\ldots,d-1$, we have
$$
I(S'_{0},\ldots,S'_{d-1})=I(S_{j},\ldots,S_{j+d-1}) ,
$$
where the quantity in the left-hand side counts the corresponding flags in $\pi$, regarded
as a copy of $K^{d+m-j}$, and these flags are partial, and not complete, unless $j=m$.

Assuming that $j<m$ (that is, $\pi$ is of dimension greater than $d$), choose a generic $d$-flat $\pi_0$ in $\pi$, so that, 
in particular, the projection of every element of $S'_{0},\ldots,S'_{d-1}$ onto $\pi_0$ 
has the same dimension as that of the original flat.  Again, if a $p$-flat is contained 
in a $q$-flat, their respective projections are also contained in one another.  
Therefore, putting $S''_{t} := \{\text{proj}_{\pi_0}(f) \mid f\in S'_{t}\}$, 
for $t=0,\ldots,d-1$, we have 
$$
I(S'_{0},\ldots,S'_{d-1}) \le I(S''_{0},\ldots,S''_{d-1}) ,
$$
where the right-hand side now counts complete flags in $\pi_0$, regarded as a copy of $K^d$.
By assumption, we have
$$
I(S''_{0},\ldots,S''_{d-1}) \le f_{d-1}(|S''_{0}|,\ldots,|S''_{d-1}|) ,
$$
and since $|S''_{t}|=|S_{j+t}|$ for each $t$, we get the asserted bound
$$
I(S_{j},\ldots,S_{j+d-1}) \le f_{d-1}(|S_{j}|,\ldots,|S_{j+d-1}|).
$$
\end{proof}

Before we go to the proof of Theorem~\ref{Flags} in arbitrary dimensions, we establish it for $d=3$, as a warm-up exercise.
\begin{theorem}\label{3D}
Let $P$ (resp., $L$, $S$) be a finite set of distinct points (resp., lines, planes) in $\R^3$ or in $\C^3$. Then
$$
I(P,L,S) = O\left(|P|^{2/3}|L|^{2/3}+|L|^{2/3}|S|^{2/3}+|P||S|+|L| \right).
$$
The bound is tight in the worst case.
\end{theorem}
\begin{proof}
Consider first the set of lines in $L$ that are incident to at most one point of $P$.  
Clearly, the number of flags involving such lines is at most $I(L,S)$, which, by 
Lemma \ref{Lift} and the Szemer\'edi--Trotter bound (Theorem~\ref{ST} or~\ref{TZ}),
is $O\left( |L|^{2/3}|S|^{2/3}+|L|+|S| \right)$.

Symmetrically, consider the set of lines in $L$ that are incident to (contained in) 
at most one plane of $S$. Again, the number of flags involving such lines is at most $I(P,L)$, 
which, by Lemma \ref{Lift} and the Szemer\'edi--Trotter bound, is 
$O\left( |P|^{2/3}|L|^{2/3}+|P|+|L| \right)$.

This leaves us with the set $L_0$ of lines that are incident to at least two points of $P$ and 
are contained in at least two planes of $S$. For each $\ell\in L_0$, let $P_\ell$ denote the set of 
points that are incident to $\ell$, and let $S_\ell$ denote the set of planes that contain $\ell$.  
By the Cauchy-Schwarz inequality, the number of flags involving the lines of $L_0$ is then at most
\begin{align*}
\sum_{\ell\in L_0}|P_\ell||S_\ell| & \leq \left(\sum_{\ell\in L_0}|P_\ell|^2\right)^{1/2}
\left(\sum_{\ell\in L_0}|S_\ell|^2\right)^{1/2} \\
& = O\left(
\left(\sum_{\ell\in L_0}\binom{|P_\ell|}2\right)^{1/2}\left(\sum_{\ell\in L_0}\binom{|S_\ell|}2\right)^{1/2} \right).
\end{align*}
The first sum can be interpreted as the overall number of pairs of points on the same line of $L_0$.
This quantity is at most $\binom{|P|}2$, because each pair of points can only lie on one common line.
Similarly, since any pair of planes of $S$ intersect in only one line, the second sum is at most 
$\binom{|S|}2$.  Therefore, 
$$
I(P,L_0,S)  = O\left( \binom{|P|}2^{1/2}\binom{|S|}2^{1/2} \right) = 
O\left((|P|^2)^{1/2}(|S|^2)^{1/2}\right) = O\left(|P||S| \right).
$$
Combining these three bounds, we get 
$$
I(P,L,S) = O\left(|P|^{2/3}|L|^{2/3}+|L|^{2/3}|S|^{2/3}+|P||S|+|L| \right),
$$
as asserted.
\end{proof}

\paragraph{Lower bounds.}
We skip the lower bound construction, because it is a special case of the construction 
for arbitrary dimension $d$, which will be described in detail below.

\begin{proof}[of Theorem \ref{Flags}]
The proof is a generalization of the proof of Theorem~\ref{3D}. It holds over both $\R$ and $\C$.

We recall the notation involving partial flags.
Let $\ff = (f_0,\ldots,f_{d-1})$ be a flag in $S_0\times \cdots\times S_{d-1}$.  
For each pair of indices $0\le i\le j\le d-1$, put $\ff_{i,j} : = (f_i,\ldots,f_j)$, 
and refer to it as a (partial) $(i,j)$-flag. In general, the $(i,j)$-flags that we will be
considering in the proof will be sub-flags of complete flags counted in $I(S_0,\ldots,S_{d-1})$.

A simple yet important observation is that, for any pair of flags $\ff$, $\ff'$, such that 
$f_{i-1}\ne f'_{i-1}$ but $f_i = f'_i$, the common $i$-flat of the pair is uniquely determined 
from $f_{i-1}\ne f'_{i-1}$: it is the unique $i$-flat that contains these two distinct $(i-1)$-flats.  
Symmetrically, if $f_{i+1}\ne f'_{i+1}$ but $f_i = f'_i$, the common $i$-flat of the pair is uniquely 
determined from $f_{i+1}\ne f'_{i+1}$: it is the unique intersection $i$-flat of these two distinct 
$(i+1)$-flats. (For an arbitrary pair $f_{i-1}\ne f'_{i-1}$ or $f_{i+1}\ne f'_{i+1}$, the common 
$i$-flat $f_i$ need not exist at all, but if it exists it must be unique.)

The theorem holds for $d=2$ (the standard Szemer\'edi--Trotter bound in Theorem~\ref{ST} or 
its complex counterpart in Theorem~\ref{TZ}),
and for $d=3$ (Theorem~\ref{3D}), but here we will only exploit its validity for $d=2$.

Assume then that $d\ge 3$, and fix some initial index $0<i<d-1$.
Put $S_{i,1}$ for the set of $i$-flats $f_i \in S_i$ such that $I(S_0,\ldots, S_{i-1},\{f_i\})\le 1$, and
$S_{i,2}$ for the set of $i$-flats $f_i\in S_i$ satisfying $I(\{f_i\}, S_{i+1},\ldots, S_{d-1})\le 1$, 
and put $S_{i,0} := S_i \setminus (S_{i,1}\cup S_{i,2})$.  We have
\begin{align}
\label {eq:sum}
I(S_0,\ldots,S_{d-1})&= I(S_0,\ldots,S_{i-1},S_{i,0},S_{i+1},\ldots,S_{d-1}) \nonumber\\
& + I(S_0,\ldots,S_{i-1},S_{i,1},S_{i+1},\ldots,S_{d-1})\\
& + I(S_0,\ldots,S_{i-1},S_{i,2},S_{i+1},\ldots,S_{d-1}).\nonumber
\end{align}
By definition, any flag
$$
(f_0,\ldots,f_{i-1},f_i,f_{i+1},\ldots, f_{d-1}) \in S_0 \times \cdots \times S_{i-1} \times S_{i,1} \times S_{i+1} \times \cdots \times S_{d-1}
$$
is uniquely determined by its suffix $(i,d-1)$-flag $(f_i,\ldots, f_{d-1})\in S_i \times \cdots \times S_{d-1}$.
Therefore,
\begin{equation}
\label {eq:s1} I(S_0,\ldots,S_{i-1},S_{i,1},S_{i+1},\ldots,S_{d-1}) \le I(S_i,\ldots, S_{d-1}).
\end{equation}
Similarly, we have
\begin{equation}
\label {eq:s2} I(S_0,\ldots,S_{i-1},S_{i,2},S_{i+1},\ldots,S_{d-1})\le I(S_0,\ldots, S_i).
\end{equation}
The overall number of flags $I:=I(S_0,\ldots,S_{i-1}, S_{i,0}, S_{i+1},\ldots, S_{d-1})$ can be written as
$$
\sum_{f_i\in S_{i,0}} I(S_0,\ldots,S_{i-1},\{f_i\}) \cdot I(\{f_i\},S_{i+1},\ldots,S_{d-1}) .
$$
Using the Cauchy-Schwarz inequality, we have
$$
I \le \left( \sum_{f_i\in S_{i,0}} I^2(S_0,\ldots,S_{i-1},\{f_i\})\right)^{1/2} \cdot 
\left( \sum_{f_i\in S_{i,0}} I^2(\{f_i\},S_{i+1},\ldots,S_{d-1})\right)^{1/2}.
$$
Each of the terms above is at most proportional to the number of pairs of $(0,i)$-prefixes or of 
$(i,d-1)$-suffixes of flags that end or start at a common $i$-flat. This follows since, for any 
$f_i\in S_{i,0}$, both $I(S_0,\ldots,S_{i-1},\{f_i\})$ and $I(\{f_i\},S_{i+1},\ldots,S_{d-1})$ are larger than $1$.
In other words, we have
\begin{align*}
\sum_{f_i\in S_{i,0}} I^2(S_0,\ldots,S_{i-1},\{f_i\}) & = O\left(
\sum_{f_i\in S_{i,0}} \binom{I(S_0,\ldots,S_{i-1},\{f_i\})}{2} \right) , \\
\sum_{f_i\in S_{i,0}} I^2(\{f_i\},S_{i+1},\ldots,S_{d-1}) & = O\left(
\sum_{f_i\in S_{i,0}} \binom{I(\{f_i\},S_{i+1},\ldots,S_{d-1})}{2} \right) . 
\end{align*}
Consider a pair $\ff$, $\ff'$ of distinct $(0,i)$-prefixes, with a common last $i$-flat $f_i$. Let $k<i$ denote the 
largest index for which $f_k\ne f'_k$ (so $0\le k\le i-1$).
Assign $(\ff,\ff')$ to their common $(k+1,i)$-flag. 
Observe that the common $(k+1)$-flat of $\ff$ and $\ff'$ is uniquely 
determined from their $(0,k)$-prefixes (in fact just from their $k$-th components).
In other words, 
$$
\sum_{f_i\in S_{i,0}} I^2(S_0,\ldots,S_{i-1},\{f_i\}) = O\left(
\sum_{f_i\in S_{i,0}} \binom{I(S_0,\ldots,S_{i-1},\{f_i\})}{2} \right)
$$
can be upper bounded by
$$
\sum_{k=0}^{i-1} I(S_{k+2},\ldots,S_i) \binom{I(S_0,\ldots,S_k)}{2} \le \sum_{k=0}^{i-1} I(S_{k+2},\ldots,S_i) I^2(S_0,\ldots,S_k) .
$$
(In this notation, the factor $I(S_{k+2},\ldots,S_i)$ is taken to be $1$ for $k=i-1$.) Indeed, ${\displaystyle \binom{I(S_0,\ldots,S_k)}{2}}$ counts all
pairs of flags in $I(S_0,\ldots,S_k)$, so it upper bounds the number of pairs that we actually want to count, namely those with different last 
(that is, $k$-th)
components. Moreover, a specific pair of partial flags that we do want to count uniquely determines the common $(k+1)$-flat of the pair, so the number of
$(0,i)$-flags that generate a specific pair of $(0,k)$-flags is at most the number $I(S_{k+2},\ldots,S_i)$ of $(k+2,i)$-flags. 

A symmetric argument applies to the suffixes term, and yields the bound
$$
\sum_{\ell=i+1}^{d-1} I(S_i,\ldots,S_{\ell-2}) I^2(S_\ell,\ldots,S_{d-1}) ,
$$
with the same convention that $I(S_i,\ldots,S_{\ell-2}) = 1$ for $\ell=i+1$. That is, we have obtained the following inequality
\begin{align*}
&I(S_0,\ldots,S_{i-1}, S_{i,0}, S_{i+1}, \ldots, S_{d-1})\\
&\le\left( \sum_{k=0}^{i-1} I^2(S_0,\ldots,S_k) I(S_{k+2},\ldots,S_i)\right)^{1/2}
\left( \sum_{\ell=i+1}^{d-1} I(S_i,\ldots,S_{\ell-2}) I^2(S_\ell,\ldots,S_{d-1}) \right)^{1/2} \\
&= O\left(\sum_{k=0}^{i-1} \sum_{\ell=i+1}^{d-1} I(S_0,\ldots,S_k) I(S_{k+2},\ldots,S_i)^{1/2}I(S_i,\ldots,S_{\ell-2})^{1/2} I(S_\ell,\ldots,S_{d-1})
\right) .
\end{align*}
We use the following (potentially crude) estimates
\begin{align*}
I(S_{k+2},\ldots,S_i)&\le I(S_{k+2},\ldots,S_{\ell-2})\\I(S_i,\ldots,S_{\ell-2})&\le I(S_{k+2},\ldots,S_{\ell-2}) ,
\end{align*}
and thus
$$
I(S_{k+2},\ldots,S_i)^{1/2} I(S_i,\ldots,S_{\ell-2})^{1/2} \le  I(S_{k+2},\ldots,S_{\ell-2}) ,
$$
where we use (an extension of) the previous notation, that $I(S_{k+2},\ldots,S_i) = 1$ 
for $k=i-1$, $I(S_i,\ldots,S_{\ell-2}) = 1$ for $\ell=i+1$, and $I(S_{k+2},\ldots,S_{\ell-2}) = 1$ 
for $k=i-1$ and $\ell=i+1$. We therefore have
\begin{align}\label{eq:s0}
I(S_0,\ldots, & S_{i-1}, S_{i,0}, S_{i+1},\ldots, S_{d-1}) \\
&= O\left(\sum_{k=0}^{i-1} \sum_{\ell=i+1}^{d-1} I(S_0,\ldots,S_k)I(S_{k+2},\ldots,S_{\ell-2}) I(S_\ell,\ldots,S_{d-1}) \right). \nonumber
\end{align}
Combining Equations~(\ref{eq:s1}),~(\ref{eq:s2}) and (\ref{eq:s0}), and substituting in Equation~(\ref{eq:sum}), we deduce the recurrence formula
\begin{align} \label{recur}
I(S_0,\ldots,S_{d-1}) & = O\Biggl(I(S_0,\ldots, S_i)+I(S_i,\ldots, S_{d-1})\Biggr.\\
&+\Biggl. \sum_{k=0}^{i-1} \sum_{\ell=i+1}^{d-1} I(S_0,\ldots,S_k)I(S_{k+2},\ldots,S_{\ell-2}) I(S_\ell,\ldots,S_{d-1}) \Biggr) . \nonumber
\end{align}
We now apply Lemma~\ref{Lift} and the induction hypothesis to each of the 
summands in the double sum in (\ref{recur}), and bound it, up to a constant factor, by
$$
\left(\sum_{(a_0,\ldots,a_k)}\prod_{i=0}^k |S_i|^{a_i}\right)
\left(\sum_{(a_{k+2},\ldots,a_{\ell-2})}\prod_{i=k+2}^{\ell-2} |S_i|^{a_i}\right)
\left(\sum_{(a_\ell,\ldots,a_{d-1})}\prod_{i=\ell}^{d-1} |S_i|^{a_i}\right),
$$
where sums are taken over tuples satisfying the requirements (i)--(iv) in the theorem statement
for the respective ranges of indices. This can be rewritten as
\begin{equation} \label{indsum}
\sum_{(a_0,\ldots,a_{d-1})}\prod_{i=0}^{d-1} |S_i|^{a_i} ,
\end{equation}
where each tuple $(a_0,\ldots,a_{d-1})$ that participates in the sum is of the form
$$
(a_0,\ldots,a_k,0,a_{k+2},\ldots,a_{\ell-2},0,a_\ell,\ldots,a_{d-1}) ,
$$
and each of its sub-tuples
$(a_0,\ldots,a_k)$, $(a_{k+2},\ldots,a_{\ell-2})$, $(a_\ell,\ldots,a_{d-1})$
satisfies (i)--(iv) for the respective ranges of indices.

We bound the two initial terms in (\ref{recur}) in a similar way, and expand each
resulting tuple $(a_0,\ldots,a_i)$ or $(a_i,\ldots,a_{d-1})$ into a full tuple by padding it
with a suffix or a prefix of $0$'s, as appropriate.

We claim that all the terms in the resulting bound involve tuples $(a_0,\ldots,a_{d-1})$ that
satisfy restrictions (i)--(iii) in the theorem statement.  
Indeed, there can be no three consecutive nonzero $a_j$'s, because this holds for the three
sub-tuples of the tuple, and the only entries that have been added are $a_{k+1}=0$ and $a_{\ell-1}=0$.  
Also, every $1$ is still preceded and followed by $0$'s, and every $2/3$ is either preceded or 
followed by another $2/3$, for similar reasons.  

Condition (iv) may be violated, but if we have a run of three or more consecutive $0$'s, we
can replace its inner portion (excluding the first and last $0$'s) by a valid sub-tuple that starts
and ends with nonzero values. For example a single inner $0$ can be replaced by $1$, two inner zeros by
a pair of $2/3$'s, three by $(1,0,1)$, four by, say, $(2/3,2/3,0,1)$, and so on. 
The product that corresponds to the modified tuple
is certainly at least as large as the original product, and this allows us to upper bound the sum
in (\ref{indsum}) by a similar sum that only involves valid tuples (that satisfy (i)--(iv) for the full range
$(0,\ldots,d-1)$). Note that this modification of the tuples $(a_0,\ldots,a_{d-1})$ might
result in tuples that appear multiple times. Nevertheless, the maximum  multiplicity of a tuple
is at most some constant that depends on $d$, so there is no effect on the asymptotic bound that we set
to establish.

Finally, substituting the modified bounds in (\ref{indsum}) into (\ref{recur}), we get the
desired bound asserted in the theorem. This establishes the induction step and thus 
completes the proof.
\end{proof}

\paragraph{Lower bounds.}
We next give examples that show that the bound in Theorem \ref{Flags} is tight in the worst case.  In particular,
for each of the terms, we can construct an example where the number of flags is  at least the bound given by that term.
First we note that the known lower bound constructions in two dimensions, e.g., the one due to Elekes~\cite{El}, can be lifted to any higher
dimension. That is, for any dimension $d\ge 3$, any $0\le i\le d-2$, and any pair of integers $m,n$ with $n^{1/2}
\le m\le n^2$, we can construct a set $S_i$ of $m$ distinct $i$-flats, and a set $S_{i+1}$ of $n$ distinct $(i+1)$
-flats in $\R^d$ or $\C^d$ with $\Theta(m^{2/3}n^{2/3})$ incidences (i.e., containments). To do so, we fix a
2-plane $\pi_0$ and place on it a set $P$ of $m$ points and a set $L$ of $n$ lines in a configuration that has
$\Theta(m^{2/3}n^{2/3})$ incidences, as in \cite{El}. Then we fix a generic $(i-1)$-flat $Q$, and note that
there exists a generic $(i+2)$-flat $R$ that contains $Q$ and $\pi_0$.
For each point $p\in P$ we construct the $i$-flat $p^*$ spanned by $p$ and $Q$, and for each line 
$\ell\in L$ we construct the $(i+1)$-flat spanned by $\ell$ and $Q$. If we choose $Q$ and $R$ 
(that is, $Q$ and $\pi_0$) sufficiently generic, we can ensure that all the flats $p^*$ are distinct 
$i$-flats, and all the flats $\ell^*$ are distinct $(i+1)$-flats. Moreover, all our flats contain $Q$ and
are contained in $R$. Finally, if $p\in P$ is incident to $\ell\in L$ then $p^*$ is contained in $\ell^*$. 
This completes this basic construction.

Recall that each of the tuples $(a_0,\ldots,a_{d-1})$ that appear in the bound of Theorem \ref{Flags} 
consists of pairs of consecutive ${2/3}$'s or solitary $1$'s separated by one or two $0$'s.
For each tuple $(a_0,\ldots,a_{d-1})$, we can construct sets $S_i$ of $i$-flats, 
of any prescribed nonzero sizes $|S_i|$, for $i=0,\ldots,d-1$, for which 
\begin{equation} \label{lb:flags}
I(S_0,\ldots,S_{d-1}) = \Omega\left(\prod_{i=0}^{d-1} |S_i|^{a_i} \right) .
\end{equation}
To do so, for each $i$ with $a_i=0$ we construct a generic $i$-flat $\pi_i$, so that these flats contain one another,
and include $\pi_i$ in $S_i$; if $|S_i|>1$, we augment it with $|S_i|-1$ additional arbitrary $i$-flats, but use
only $\pi_i$ in our construction.
For the remaining indices, if $a_i=1$, we take $S_i$ to be an arbitrary collection of $|S_i|$ $i$-flats,
all contained in $\pi_{i+1}$ and containing $\pi_{i-1}$ (whenever applicable), recalling that we must have
$a_{i-1}=a_{i+1}=0$. Finally, if $a_i=a_{i+1}={2/3}$, we apply the lifting of the planar lower bound
construction, as described above, to obtain sets $S_i$ of $|S_i|$ $i$-flats, and $S_{i+1}$ of $|S_{i+1}|$
$(i+1)$-flats, all contained in $\pi_{i+2}$ and containing $\pi_{i-1}$, with
$\Theta\left(|S_i|^{2/3}|S_{i+1}|^{2/3}+|S_i|+|S_{i+1}|\right)$ incidences (again, we must have $a_{i-1}=a_{i+2}=0$).
The overall number of flags in $\prod_{i=0}^{d-1} S_i$ clearly satisfies (\ref{lb:flags}).

\paragraph{Remark.}
A simple extension of Theorem \ref{Flags} yields similar worst-case tight bounds not only for complete flags 
but also for partial flags, namely, flags that are sequences of flats, of increasing dimensions, so that
each flat is contained in the next one, but where only some of the dimensions are being used. For example,
in the notation $I(S_0,S_1,S_3,S_4,S_5,S_7,S_8)$, $S_i$ is a finite set of distinct $i$-flats, for the indices
$i$ appearing in the notation, and $I(S_0,S_1,S_3,S_4,S_5,S_7,S_8)$ is the number of sequences
$(f_0,f_1,f_3,f_4,f_5,f_7,f_8)$, where $f_i\in S_i$ for each $i$, and $f_i\subset f_j$ for every pair $i<j$. 

To obtain a bound on the number of such partial flags, we simply break the sequence of indices into
maximal runs of consecutive indices, and apply the bound of Theorem~\ref{Flags} to each run separately.
Concretely, we obtain the following corollary.
\begin{corollary} 
Let $\sigma$ be a nonempty subsequence of $(0,\ldots,d-1)$, and write it as a
concatenation of maximal runs of consecutive indices, in the form
$$
\sigma = (i_1,\ldots,j_1,i_2,\ldots,j_2,\ldots,i_k,\ldots,j_k) ,
$$
where $i_t\le j_t$ for each $t$ and $i_{t+1} \ge j_t+2$, for $t=0,\ldots,k-1$.
For each $i\in\sigma$, let $S_i$ be a set of distinct $i$-flats in $\R^d$, and let
$I((S_t)_{t\in\sigma})$ denote the number of tuples $(f_t)_{t\in\sigma}$, where
$f_t\in S_t$ for each $t\in\sigma$, and $f_t\subset f_{t'}$ for each pair of consecutive
elements $t,t'\in\sigma$. Then we have
$$
I((S_t)_{t\in\sigma}) \le I(S_{i_1},\ldots,S_{j_1}) 
I(S_{i_2},\ldots,S_{j_2}) \cdots 
I(S_{i_k},\ldots,S_{j_k}) ,
$$
where each of the factors in the right-hand side is defined as in Lemma~\ref{Lift},
and can be bounded by the bound given in that lemma. The resulting bound is tight
in the worst case.
\end{corollary}
For instance, for the example just given, we have
$$
I(S_0,S_1,S_3,S_4,S_5,S_7,S_8) \le I(S_0,S_1)I(S_3,S_4,S_5)I(S_7,S_8) ,
$$
from which an upper bound can be worked out, as above.
The tightness of the bound can be argued via a modified version of the construction
given above for the lower bound in Theorem~\ref{Flags}.

\section{Variants of the Incidence Problem in $\R^3$}\label{Variants}

In this section, we consider three variants of the classical incidence problem studied in the previous section.  
As already discussed in the introduction, the motivation behind these problems is to show that, in many natural 
cases, the reduction of incidence problems in higher dimensions to the planar Szemer\'edi--Trotter bound, 
as done in the previous section, is too weak and leads to inferior bounds. The problems studied in this
section provide new examples where sharper bounds can be established, adding to the growing body of such 
results, initiated by Guth and Katz in~\cite{GK1,GK2}.  In all of these, we use the following polynomial
partitioning lemma due to Guth and Katz:
\begin{lemma}[Guth and Katz \cite{GK1}]\label{Partition}
If $S$ is a set of points in $\R^d$ and $D$ is an integer, then there is a nonzero $d$-variate polynomial $Q$ of degree at most $D$ so that $\R^d\setminus Z(Q)$ is the union of $O(D^d)$ open connected components, each of which contains $O\left(\frac{|S|}{D^d}\right)$ points of $S$.
\end{lemma}

\subsection{Not too many lines are both concurrent and coplanar}

First, we consider a generalized version of a problem studied by Sharir and Welzl in \cite{SW}, 
involving the number of incidences between a set $P$ of points and a set $L$ of ``light-like'' lines
(also referred to as equally inclined lines) in $\R^3$; these are lines that are parallel to some fixed 
double cone (such as $z^2=x^2+y^2$), or, equivalently, form a fixed angle with, say, the $z$-axis.
Sharir and Welzl~\cite{SW} gave an example with 
$\Theta(|P|^{2/3}|L|^{1/2})$ incidences.  They also proved that the number of incidences is
$O(|P|^{3/4}|L|^{1/2}\log|P|+|P|+|L|)$ and $O(|P|^{4/7}|L|^{5/7}+|P|+|L|)$. 
In Elekes et al.~\cite{EKS}, the latter bound was improved to $O(|P|^{1/2}|L|^{3/4}+|P|+|L|)$.  
Here we use the polynomial cell decomposition technique, in an analysis that resembles that of
Guth and Katz~\cite[Theorem 4.5]{GK2} and of Elekes et al.~\cite{EKS}, 
but is considerably simpler, to give a different proof of this latter
bound in a more general context.

\begin{theorem}\label{Light}
Let $b$ be a constant, and let $P$ be a set of points and $L$ a set of lines in $\R^3$,
so that each point of $P$ has at most $b$ coplanar lines of $L$ through it. Then
$$
I(P,L) = O\left(|P|^{1/2}|L|^{3/4}+|P|+|L|\right) ,
$$
where the constant of proportionality depends on $b$.
\end{theorem}
\begin{proof}
As noted, the proof is a simplified version of the proof of Guth and Katz~\cite[Theorem 4.5]{GK2}.

Put $m=|P|$ and $n=|L|$.
We may ignore points that are incident to at most $b$ lines, as they contribute
a total of at most $bm=O(m)$ incidences. Let $P$ denote the set of remaining points.
Then each point of $P$ is a \emph{joint} of $L$, that is, it is incident to at least 
three non-coplanar lines of $L$. As shown by Guth and Katz~\cite{GK1}, we have $m=|P|=O(n^{3/2})$.

Assume that $m\ge n^{1/2}$ (otherwise we have the trivial bound $I(P,L)=O(n)$ from the Szemer\'edi--Trotter theorem).
Apply Lemma \ref{Partition} to $P$, to obtain a polynomial $f\in\R[x,y,z]$ of degree $D=am^{1/2}/n^{1/4}$, for some
sufficiently small constant $a$, so that the complement of the zero set $Z(f)$ of $f$ consists of $O(D^3)$ open connected cells,
each containing at most $O(m/D^3)$ points of $P$.

Now we show that the number of incidences involving lines
that are not fully contained in $Z(f)$ is $O(m^{1/2}n^{3/4}+m)$.  Let $m_i$ denote the number
of points contained in cell $O_i$, and $n_i$ denote the number of lines intersecting $O_i$.  Note
that each line can intersect $Z(f)$ at most $D$ times, so it can intersect at most $D+1$ components,
i.e. $\sum_i n_i\leq n(D+1)$.  Using Holder's inequality and the Szemer\'edi-Trotter bound, the
total number of incidences is then bounded (up to a constant factor) by
\begin{align}\label{noind}
\sum_i\left(m_i^{2/3}n_i^{2/3}+m_i+n_i\right)&\leq\left(\sum_im_i^2\right)^{1/3}
\left(\sum_in_i\right)^{2/3}+m+(D+1)n\\&=\left(D^3\cdot O\left(\frac m{D^3}\right)^2
\right)^{1/3}\left(O(nD)\right)^{2/3}+m+O(nD)\nonumber\\&=O(m^{2/3}n^{2/3}D^{-1/3}
+m+nD)=O(m^{1/2}n^{3/4}+m).\nonumber
\end{align}
(Note that this also bounds the number of incidences between lines of this kind and points on $Z(f)$.)

Let $P_0 := P\cap Z(f)$, and let $L_0$ denote the subset of the lines $\ell\in L$ that
are fully contained in $Z(f)$. It remains to bound $I(P_0,L_0)$.

We may assume, without loss of generality, that $f$ is square-free.
Consider a point $p\in P_0$ that is a non-singular point of $f$. Then all the lines
of $L_0$ that are incident to $p$ lie in the tangent plane to $Z(f)$ at $p$, and,
by assumption, there are at most $b$ such lines. Hence, the non-singular points in $P_0$
contribute a total of at most $bm=O(m)$ to $I(P_0,L_0)$. 
Let $P_0^s$ denote the set of the points in $P_0$ that are singular points of $Z(f)$. 
It now remains to bound the quantities $I(P_0^s,L_0)$.

We may ignore lines of $L_0$ that are incident to at most $D$ points of $P_0^s$, because they 
contribute to $I(P_0^s,L_0)$ a total of at most $nD=O(m^{1/2}n^{3/4})$ incidences.  
Any remaining line is fully contained in the common zero set $Z(f,f')$ of $f$ and $f'$, 
where $f'$ is any of the first-order partial derivatives of $f$ that is not identically zero.
As argued in \cite{EKS,GK1}, by B\'ezout's Theorem, the number of such lines is at most $D(D-1)<D^2$.
By choosing $a$ sufficiently small, we have 
$D^2 = a^2 m/n^{1/2} < n/2$ (because $m=O(n^{3/2})$), and we can bound the number of remaining
incidences by the maximum number of incidences involving (at most) $m$ points and (at most) $n/2$ lines.

By a simple inductive argument\footnote{%
  Clearly, the condition that each point is incident to at most $b$ coplanar lines continues to hold 
  in the inductive subproblems.}
on $|L|$ we can show that
$$
I(P,L)\leq C\left(|P|^{1/2}|L|^{3/4}+|P|+L\right)
$$
for a suitable sufficiently large constant $C$.  Indeed, for the induction step, writing $m'=|P_0^s|$,
the bounds collected so far in (\ref{noind}) add up to at most $c(m^{1/2}n^{3/4}+(m-m')+n)$, for
some suitable constant $c$.  Adding this to the inductive bound for the remaining incidences, we get the
bound
\begin{align*}
I(P,L)&\leq c\left(m^{1/2}n^{3/4}+(m-m')+n\right)+C\left(m^{1/2}(n/2)^{3/4}+m'+(n/2)\right)\\&\leq C(m^{1/2}n^{3/4}+m+n)
\end{align*}
by choosing $C$ sufficiently large.
\end{proof}

Clearly, the above theorem applies to light-like lines, with $b=2$, because
the double cone centered at any point can only intersect a plane through that point 
in at most two lines.  The theorem can also apply when we restrict the direction of the lines in
a different way. For example, we recall the identification of lines in $\R^3$ through the origin
with points in $\R\mathbb P^2$ (representing the direction of the line), and
take some irreducible nonlinear algebraic curve $\gamma$ of constant degree $b$
in $\R\mathbb P^2$, and restrict the directions of our lines to lie on this curve. 
Then at most $b=O(1)$ lines through any point can be coplanar (since the directions of coplanar lines
through the origin lie on a line in $\R\mathbb P^2$, and such a line can intersect 
a curve of degree $b$ in at most $b$ points).  The light-like lines arise when $\gamma$ is a circle.

\subsection{Incidences with Legendrian lines}

We call a line in $\R^3$ \emph{Legendrian} (see, e.g., Buczy\'nski \cite{Buc} for a more general
discussion of Legendrian subvariaties)
if it is orthogonal to the vector field $\F(x,y,z) := (y,-x,1)$ at each of its points. 
Legendrian lines arise as a special case of generalized flags, where the flags are determined 
by the action of the symplectic group $Sp(4)$ instead of the general linear group $GL_4(\R)$.
See the concluding section for details concerning this interpretation. For readers unfamiliar
with this theory, we stress that this is only a motivation for considering Legendrian lines.
Once the motivation is accepted, we face a standard (and in our opinion interesting)
incidence problem, involving a special class of lines in three dimensions.

We explore some geometric properties of Legendrian lines. 

First, we note that if a line $(a+ut,b+vt,c+wt)$ is orthogonal to $\F(a,b,c) = (b,-a,1)$ 
at one point $(a,b,c)$, it is orthogonal to $\F$ everywhere, because the scalar product
$(u,v,w)\cdot(b+vt,-a-ut,1) = bu-av+w$ is independent of $t$.
This means that for every point $(a,b,c)$, we have an associated plane $\pi(a,b,c)$, 
namely, the plane through $(a,b,c)$ with normal vector $(b,-a,1)$, so that every 
line in $\pi(a,b,c)$ through $(a,b,c)$ is Legendrian.  

Moreover, let $\pi$ be any non-vertical plane, write its normal as $(b,-a,1)$, and
let $p$ be the unique point on $\pi$ with $x$-coordinate $a$ and $y$-coordinate $b$.
Then, by the preceding argument, every line incident to $p$ and contained in $\pi$ is 
Legendrian. On the other hand, any other point $q\ne p$ contained in $\pi$ is incident
to exactly one Legendrian line, which is the line through $q$ orthogonal to both
$(b,-a,1)$ and to $(q_y,-q_x,1)$. We refer to $p$ as the \emph{Legendrian point} of $\pi$. 

We now apply these facts, in a manner that resembles the proof of Theorem \ref{Light},
and establish the following upper bound.
\begin{theorem}\label{Legendrian}
Let $P$ be a finite set of distinct points and $L$ a finite set of distinct non-vertical 
Legendrian lines in $\R^3$.  Then
$$
I(P,L) = O\left(|P|^{1/2}|L|^{3/4}+|P|+|L|\right).
$$
\end{theorem}
\begin{proof}
Assume first that $n^{1/2}\le m\le n^{3/2}$ (when $m\le n^{1/2}$ we have the trivial bound $I(P,L) = O(n)$). 
Apply Lemma \ref{Partition} to $P$, to obtain a polynomial  $f\in R[x,y,z]$ of degree $D=am^{1/2}/n^{1/4}$,
for some suitable constant $a$, so that the complement 
of the zero set $Z(f)$ of $f$ consists of $O(D^3)$ open connected cells,
each containing at most $O(m/D^3)$ points of $P$. Arguing as in the preceding proof, it is easy to 
show that the number of incidences involving lines that are not fully contained in $Z(f)$ is $O(m^{1/2}n^{3/4})$.

Let $P_0:=P\cap Z(f)$, and let $L_0$ denote the subset of lines $\ell \in L$ that are 
fully contained in $Z(f)$. The lines in $L\setminus L_0$ contribute at 
most $O(nD)=O(m^{1/2}n^{3/4})$ incidences with $P_0$ (since each of them crosses $Z(f)$ in at most $D$ points), 
so it only remains to bound $I(P_0,L_0)$. 

As above, we may assume, without loss of generality, that $f$ is square-free, and we decompose it
into distinct irreducible factors $f_1,\ldots,f_k$.
Assign each point $p\in P_0$ (resp., line in $L_0$) to the first zero set $Z(f_i)$
such that $p \in Z(f_i)$ (resp., $\ell \subset Z(f_i)$). As is easily seen, for each $\ell \in L_0$, 
the number of incidences $(p, \ell)$, where $p\in P_0$ is not assigned to the same $f_i$ as $\ell$ is at 
most $D$, yielding a total of at most $nD=O(m^{1/2}n^{3/4})$ incidences.  Therefore, it suffices to bound 
incidences involving points and lines assigned to the same component $Z(f_i)$. 

First, suppose that $Z(f_i)$ is a plane. Let $p \in P_0$ be assigned to $Z(f_i)$. As observed above, 
unless $p$ is the Legendrian point of $Z(f_i)$, $p$ is incident to at most one Legendrian line 
contained in $Z(f_i)$; if $p$ is the Legendrian point, we bound the number of its incidences by the 
trivial bound $n$. Summing over the at most $D$ planar components $Z(f_i)$, we obtain  a total of $O(m+nD)$ incidences.

If $Z(f_i)$ is not a plane, the points $p\in P_0$ that are assigned to $Z(f_i)$ and are incident to at most 
two lines of $L_0$ yield at most $2m$ incidences. The other points of $Z(f_i)$ are either \emph{singular}
points or \emph{flat} points of $Z(f_i)$; as in Guth and Katz~\cite{GK1} and in Elekes et al.~\cite{EKS},
a non-singular point of $Z(f_i)$ is said to be (linearly) flat if it is incident to (at least) three lines fully 
contained in $Z(f_i)$. As argued in these works, each singular point lies on the one-dimensional singular 
locus of $Z(f_i)$, a curve of degree smaller than $\deg(f_i)^2$, and 
each flat point lies (in $Z(f_i)$ and) in the common zero set of three polynomials, each of degree smaller than $3\deg(f_i)$. 
Moreover (as shown in \cite{EKS,GK1}), since $Z(f_i)$ is not a plane, at least one of these polynomials 
does not vanish identically on $Z(f_i)$, and so the flat points lie on a curve of degree at most 
$3\deg(f_i)^2$. A line $\ell \in L_0$ assigned to $Z(f_i)$ that is not contained in either of this pair
of curves yields at most $4\deg(f_i)$ incidences (because each such incidence is a proper crossing 
of the line with either the zero set of one of the partial derivatives of $f_i$ or the zero set of
one of the ``flatness polynomials'' mentioned above. Summing over all components, this results in
a total of at most $4nD$ incidences. The remaining lines assinged to $Z(f_i)$ are contained in an
algebraic curve (the union of the two curves) of degree at most $4\deg(f_i)^2$, so there can be at most these many lines of the
latter kind (see, e.g., \cite{EKS,GK1}). Summing over all (non-planar) components, we get a total
of at most $4D^2$ lines, which, with a suitable choice of the constant $a$ in the definition of $D$,
is at most, say, $n/2$ (since $m\le n^{3/2}$). An induction applied to these remaining lines finishes 
the proof for this case, just as in Guth and Katz~\cite{GK2}, and in the preceding proof.

Finally, when $m > n^{3/2}$, we apply the same reasoning, taking the degree of $f$ to be
$D=an^{1/2}$, for a sufficiently small constant $a$. The analysis proceeds more or less identically 
to the case $m\le n^{3/2}$, as it does in the cited works \cite{EKS,GK2}.

In both cases, we thus get the asserted bound
$$
I(P,L) = O\left(|P|^{1/2}|L|^{3/4}+|P|+|L|\right).
$$
\end{proof}

We do not know whether the bound in Theorem~\ref{Legendrian} is tight, and leave this as an open question
for further research.

\subsection{Flags with limited point/plane incidences with a line}

Now we move on to our third variant of the three-dimensional incidence problem.  
Here we consider the number of flags determined by sets of points, lines, and planes in $\R^3$
that satisfy the constraints in the following theorem.  

\begin{theorem}\label{3D'}
Let $P$ (resp., $L$, $S$) be a finite set of distinct points (resp., lines, planes) in $\R^3$, 
so that each line of $L$ contains at most $b$ points of $P$ and is contained in at most $b$ planes of $S$,
for some parameter $b$. Then
\begin{align*}
I(P,L,S) & = O\left( \min \{b^2|L|, \right. \\
& \left. (|P|+|S|)|L|^{1/2} + (|P||S|^{1/2}+|S||P|^{1/2})\log b + (|P|+|S|)b \right) . 
\end{align*}
In particular, when each of $|P|$, $|L|$, and $|S|$ is at most some number $N$, we have
$$
I(P,L,S) = O\left(\min \{b^2N,\;N^{3/2}\log b+bN\}\right) .
$$
\end{theorem}
\begin{proof}
For each line $\ell\in L$, let $P_\ell$ denote the set of points of $P$ that are
incident to $\ell$, and let $S_\ell$ denote the set of planes of $S$ that contain $\ell$.
By assumption, we have $p_\ell:=|P_\ell|$, $s_\ell:=|S_\ell|\le b$. Our goal is to bound
$$
I(P,L,S) = \sum_{\ell\in L} p_\ell s_\ell .
$$
Let $N_{k,l}$ (resp., $N_{\ge k,\ge l}$) denote the number of lines $\ell\in L$ for which
$p_\ell = k$ and $s_\ell = l$ (resp., $p_\ell \ge k$ and $s_\ell \ge l$).
Then we have
$$
I(P,L,S) = \sum_{k=1}^b \sum_{l=1}^b kl N_{k,l} . 
$$
As is easily checked, we have
$$
N_{k,l} = N_{\ge k,\ge l} - N_{\ge k+1,\ge l} - N_{\ge k,\ge l+1} + N_{\ge k+1,\ge l+1} .
$$
Substituting this and rearranging the preceding sum, we get
\begin{align*} 
I(P,L,S) & = \sum_{k=1}^b \sum_{l=1}^b kl N_{k,l} \\
& = \sum_{k=1}^b \sum_{l=1}^b kl \Bigl( N_{\ge k,\ge l} - N_{\ge k+1,\ge l} - N_{\ge k,\ge l+1} + N_{\ge k+1,\ge l+1} \Bigr) \\
& = \sum_{k=1}^b \sum_{l=1}^b \Bigl( kl-(k-1)l-k(l-1)+(k-1)(l-1) \Bigr) N_{\ge k,\ge l} \\
& = \sum_{k=1}^b \sum_{l=1}^b N_{\ge k,\ge l} .
\end{align*}
We proceed to estimate $N_{\ge k,\ge l}$. We fix $k$ and $l$, and assume, without loss of generality, 
that $l\le k$. Otherwise, we apply a point-plane duality to 3-space, turn $P$ into a set of $|P|$ dual 
planes, $S$ into a set of $|S|$ dual points, and $L$ into a set of $|L|$ dual lines, so that a point-line
(resp., line-plane) incidence in the primal space become a plane-line (resp., line-point) incidence in 
dual space, and vice versa, so the number of flags in the dual setting is equal to that in the primal 
setting. In this manner the roles of $k$ and $l$ are interchanged, and this justifies our assumption.

In other words, it suffices to bound the sum
\begin{equation} \label{ipls}
I'(P,L,S) := \sum_{k=1}^b \sum_{l=1}^k N_{\ge k,\ge l} .
\end{equation} 
We thus fix $1\le l\le k$, and consider the set
$$
L_0 := \{\ell\in L \mid p_\ell\ge k \text{ and } s_\ell \ge l \} ;
$$ 
these are the lines counted in $N_{\ge k,\ge l}$. 

We claim that no plane can contain more than $B:= |S|/(l-1)$ lines of $L_0$. Indeed, fix a plane $\pi$
(not necessarily in $S$), and observe that each line $\ell\in L_0$ that is contained in $\pi$ is also 
contained in at least $l-1$ planes of $S$ (other than $\pi$), and that all these planes are distinct, over all lines
$\ell\in\pi$, so the claim follows. (For $l=1$ the bound is simply $|S|$.)

We can therefore apply the incidence bound of Guth and Katz~\cite{GK2}, as mentioned in (\ref{eq:gk})
in the introduction, to obtain
\begin{align*}
I(P,L_0) & = O\Bigl( |P|^{1/2}|L_0|^{3/4} + |P|^{2/3}|L_0|^{1/3}B^{1/3} + |P| + |L_0| \Bigr) \\
& = O\Bigl( |P|^{1/2}|L_0|^{3/4} + |P|^{2/3}|S|^{1/3}|L_0|^{1/3}/l^{1/3} + |P| + |L_0| \Bigr) .
\end{align*}
On the other hand, by definition, we have $I(P,L_0) \ge k|L_0|$. This implies that
$$
k|L_0| = O\Bigl( |P|^{1/2}|L_0|^{3/4} + |P|^{2/3}|S|^{1/3}|L_0|^{1/3}/l^{1/3} + |P| + |L_0| \Bigr) ,
$$
or, assuming that $k$ is at least some sufficiently large constant $k_0$,
\begin{equation} \label{nkl}
|L_0| = O\left( \frac{|P|^2}{k^4} + \frac{|P||S|^{1/2}}{k^{3/2}l^{1/2}} + \frac{|P|}{k} \right) .
\end{equation}

We now return to the sum $I'(P,L,S)$ in (\ref{ipls}). For technical reasons that will become
clear shortly, we fix a threshold parameter $t<b$ (larger than the constant $k_0$ used above), 
and split the sum in (\ref{ipls}) as
\begin{align}\label{tsplit}
\sum_{k=1}^t \sum_{l=1}^k N_{\ge k,\ge l} +
\sum_{k=t+1}^b \sum_{l=1}^k N_{\ge k,\ge l} .
\end{align}
In the first sum we use the trivial bound $N_{\ge k,\ge l} \le |L|$, making the sum bounded by
$O(|L|t^2)$. We then substitute the bound in (\ref{nkl}) into the second sum, and get 
\begin{align*}
\sum_{k=t+1}^b \sum_{l=1}^k N_{\ge k,\ge l} & =
O\Bigl( \sum_{k=t+1}^b \sum_{l=1}^k 
\Bigl( \frac{|P|^2}{k^4} + \frac{|P||S|^{1/2}}{k^{3/2}l^{1/2}} + \frac{|P|}{k} \Bigr) \Bigr) \\
& = O\Bigl( \sum_{k=t+1}^b 
\Bigl( \frac{|P|^2}{k^3} + \frac{|P||S|^{1/2}}{k} + |P| \Bigr) \Bigr) \\
& = O\Bigl( \frac{|P|^2}{t^2} + |P||S|^{1/2}\log b + |P|b \Bigr) .
\end{align*}
In the complementary case, where $k\le l$, we get the symmetric bound (where the roles of $P$ and $S$ are interchanged)
$$
\sum_{l=1}^b\sum_{k=1}^lN_{\ge k,\ge l} =
O\Bigl(|L|t^2 + \frac{|S|^2}{t^2} + |S||P|^{1/2}\log b + |S|b \Bigr) .
$$
Choosing ${\displaystyle t=\left(\frac{|P|^2+|S|^2}{|L|}\right)^{1/4}}$,
the preceding reasoning is easily seen to yield the bound 
$$
I(P,L,S) = O\left(
(|P|+|S|)|L|^{1/2} + (|P||S|^{1/2}+|S||P|^{1/2})\log b + (|P|+|S|)b \right) .
$$
When ${\displaystyle b < \left(\frac{|P|^2+|S|^2}{|L|}\right)^{1/4}}$, 
the second sum in (\ref{tsplit}), and its counterpart in the case $k\leq l$, become vacuous, and we get the bound
$I(P,L,S) = O(|L|b^2)$.  This completes the proof of the theorem.
\end{proof}

\paragraph{Lower bounds.}
We now show that the bound in Theorem \ref{3D'} is almost tight in the worst case (except for the $\log b$ factor),
for sets $P$, $L$, $S$ with $|P|=|L|=|S|=N$; actually, because of the asymptotic nature of the bound, it suffices to
consider sets whose sizes are (at most) proportional to $N$.  
Recall that in this case the bound that we want to establish is
$I(P,L,S) = O\left(\min \{b^2N,\;N^{3/2}\log b+bN\}\right)$.

\medskip

\noindent
(i) To obtain a configuration with $\Theta(bN)$ flags, take $N/b$ parallel lines, 
each incident to $b$ points and contained in $b$ planes.  In total we have $N$ points, 
at most $N$ lines, and $N$ planes, with $(N/b)\cdot b\cdot b= bN$ flags.

\medskip

\noindent
(ii) In the case under consideration, we have
${\displaystyle t=\left(\frac{|P|^2+|S|^2}{|L|}\right)^{1/4}} = \Theta(N^{1/4})$.
Assume that $b \ge t = \Omega\left( N^{1/4}\right)$. We adapt the lower bound construction
of Guth and katz~\cite{GK2}, which in turn is based on the aforementioned construction
of Elekes~\cite{El}. Concretely, we fix two integer parameters $k,l$, and construct the
$k\times 2kl\times 2kl$ integer grid 
$$
P = [1,k]\times [1,2kl]\times [1,2kl] ,
$$
(where $[1,m]$ is a shorthand notation for the set
of integers $\{1,2,\ldots,m\}$);
we have $|P|=4k^3l^2$. For the set of lines we take
$$
L = \{ y=ax+b,\;z=cx+d \mid a,c\in [1,l],\; b,d\in [1,kl] \};
$$
we have $|L|=k^2l^4$. By taking $k=l^2$ and putting $N=l^8$, we get $|P|=4N$ and $|L|=N$.
The sets $P$ and $L$ have the property that each line of $L$ is incident to exactly
$k$ points of $P$.

So far, this is the construction used in \cite{GK2}. We now construct the set $S$
of planes. For each point $p=(x,y,z)$ in $P$, we take $L_p$ to be the set of all 
lines of $L$ that pass through $p$. We note that, for a constant fraction of the
choices of $p\in P$, and $a,c\in [1,l]$, the integers $b=y-ax$, $d=z-cx$ belong to $[1,kl]$,
and the quadruple $(a,b,c,d)$ determines a line of $L$ that passes through $p$.
In other words, a constant fraction of points $p\in P$ are incident to $\Omega(l^2)$ lines of $L$ 
(the concrete constants of proportionality can easily be worked out).

For each point $p$ and each pair of distinct lines $\ell,\ell'\in L_p$, 
we form the plane spanned by $\ell$ and $\ell'$ and include it in $S$. Clearly,
the same plane in $S$ might arise for multiple choices of $p$, $\ell$, and $\ell'$.
To estimate $|S|$ we proceed as follows. 
For a constant fraction of the choices of $p\in P$, $a,c,a',c'\in [1,l]$, 
so that the pairs $(a,c)$ and $(a',c')$ are distinct, we get a pair of lines 
$\ell$, $\ell'\in L_p$, parallel to the respective vectors $(1,a,c)$, $(1,a',c')$.
The plane $\pi$ that they determine has a normal 
$$
{\bf n}_\pi = (1,a,c)\times (1,a',c') = (ac'-a'c,\;c-c',\;a'-a) ,
$$
so
${\bf n}_\pi$ is determined just from the choices of $a,c,a',c'\in [1,l]$, so there
are at most $l^4$ different normals. The number of planes $\pi$ with a given normal
$\bf n$ is at most the number of distinct scalar products $p\cdot {\bf n}$, over
$p\in P$. Noting that the coordinates of any normal lie in $[-l^2,l^2]\times [-l,l]\times [-l,l]$
(and those of the points of $P$ in $[1,k]\times [1,2kl]\times [1,2kl]$), the value of each scalar product
is a whole number of maximum absolute value $O(kl^2)$. Hence
$$
|S| = O(l^4\cdot kl^2) = O(kl^6) = O(N) .
$$
We now estimate $I(P,L,S)$. Pick a line $\ell\in L$, given by $y=ax+b$, $z=cx+d$.
Pick any point $p\in P\cap \ell$, and let $(a',c')$ be a pair in $[1,l]^2$ distinct
from $(a,c)$. For a constant fraction of these choices, the line passing through
$p$ in direction $(1,a',c')$ belongs to $L$, and spans with $\ell$ a plane in $S$.
An easy calculation shows that the normal vectors
\begin{align*}
{\bf n'} & = (1,a,c)\times (1,a',c') = (ac'-a'c,\;c-c',\;a'-a) , \\
{\bf n''} & = (1,a,c)\times (1,a'',c'') = (ac''-a''c,\;c-c'',\;a''-a)  
\end{align*}
have different directions if and only if the vectors $(c'-c,\;a'-a)$ and $(c''-c,\;a''-a)$ 
are not parallel. A standard argument in number theory 
(see, e.g., \cite{Ed} for a similar application thereof)
shows that there are $\Theta(l^2)$ choices of $a',c'$ that yield normals with distinct directions.
In other words, a constant fraction of lines of $L$ are incident to $\Theta(l^2)$ planes of $S$,
and no line is incident to more than $O(l^2)$ planes. (The latter statement follows since, for a
fixed line $\ell$, each plane $\pi$ containing $\ell$ is spanned by $\ell$ and another line
$\ell'$, and the direction of $\ell'$ suffices to determine $\pi$; since there are $O(l^2)$ 
such directions, the claim follows.)
Since each line of $L$ is incident to exactly $k$ points of $P$, the number of flags satisfies
the bound
$$
I(P,L,S) = \Omega(kl^2|L|) = \Omega(N^{3/2}) .
$$
Since both $k$ and $l^2$ are $O(N^{1/4})$, it follows that, with a suitable choice
of the constant of proportionality, each line is incident to at most $b=\Omega(N^{1/4})$ points of $P$ 
and is contained in at most $b$ planes of $S$. This completes the analysis for this case.

\medskip

\noindent
(iii) Finally, assume that $b\le t = O\left( N^{1/4}\right)$; that is, assume that $b^4 = O(N)$.
Apply the preceding construction with $b^4$ instead of $N$, and construct $N/b^4$ independent copies
thereof. We get in total $\Theta(N)$ points, lines, and planes, each line is incident to at most $b$
points and is contained in at most $b$ planes, and the number of flags is 
$\Omega((b^4)^{3/2}\cdot (N/b^4)) = \Omega(Nb^2)$, as required.

\section{Discussion} \label{Other}

\subsection{On generalized flags and Legendrian lines}

In this section we briefly review the following generalization of flags, which serves as a motivation
for studying incidences with Legendrian lines (Theorem~\ref{Legendrian}). This is a topic that
has received considerable attention in the theory of Lie groups. This discussion is fairly disjoint 
from the main study of the paper, and is included only to put things into perspective, as well
as in the hope that it will help to motivate additional kinds of incidence problems. We will not 
go into details concerning this topic; they can be found, e.g., in Baston and Eastwood~\cite{BE}.

Consider first a standard flag $\ff=(f_1,\ldots,f_d)$ in $\R^d$, where each $f_i$ is an $i$-dimensional subspace, 
and $f_i\subset f_{i+1}$ for $i=1,\ldots,d-1$.  (The flags we refer to in this discussion are the projective
versions of the affine flags which we have dealt with thus far.)  We can write $\ff$ as a full-rank $d\times d$ matrix
$M_\ff$, with rows $v_1,\ldots,v_d$, such that $f_i$ is the span of $v_{d-i+1},\ldots,v_d$
for each $i$; clearly, such a representation is not unique. The \emph{stabilizer} of this 
representation is the set of all non-singular upper triangular $d\times d$ matrices, in the 
sense that, for each such matrix $T$ and for each flag $\ff$, if $\ff$ is represented by 
some matrix $M$ then it is also represented by $TM$. We can thus interpret the set of all
flags in $\R^d$, the so-called (standard) \emph{flag variety},
as the quotient group $GL(d)/UT(d)$, where $GL(d)$ is the general linear group in 
$d$ dimensions, and $UT(d)$ is the set of all $d\times d$ invertible upper triangular matrices.  

Similarly, we may consider the following notion of \emph{generalized flags}.
Given a Lie group $G$ (in the above example, $G=GL(d)$, but we may also consider 
the symplectic group $Sp(d)$, or the orthogonal group $O(d)$, etc.),
a \emph{Borel} subgroup is a maximal Zariski closed and connected solvable subgroup. 
When $G=GL(d)$, (it is actually sufficient to restrict ourselves to the special 
linear group $SL(d)$), the subgroup $UT(d)$ of invertible upper triangular matrices is a 
Borel subgroup. A \emph{parabolic} subgroup of $G$ is any subgroup of $G$ that 
contains a Borel subgroup of $G$. 

When $G=GL(d)$, the flag variety of $G$ is, as mentioned above, the quotient group $GL(d)/UT(d)$.
In the general setup, if $G$ is a semisimple Lie group, the \emph{generalized flag variety}
for $G$ of type $P$ is $G/P$, where $P$ is a parabolic subgroup of $G$. The points of $G/P$ 
can be interpreted in terms of the more intuitive notion of ``flags''.
We demonstrate this generalization by considering two examples: the split orthogonal group 
$O(2,2)$, and the symplectic group $Sp(4)$; in  both cases we consider a nested sequence of real 
vector subspaces in $\R^4$ that are all annihilated by a cetain symmetric or symplectic form.

In the case of the group $O(2,2)$, a flag is a sequence of nested subspaces of $\R^4$,
so that some fixed symmetric form of signature $(2,2)$ vanishes on all its proper subspaces. 
This can be shown to be possible only for one- and two-dimensional subspaces.  
Identifying $\R^4\setminus\{0\}$ with $\R\mathbb P^3$, this will correspond to 
point-line pairs in $\R^3$. Concretely, if we simply choose the form to have the matrix
$$
\begin{bmatrix}1&0&0&0\\0&1&0&0\\0&0&-1&0\\0&0&0&-1\end{bmatrix},
$$
then a vector $(a,b,c,d)$ is annihilated by (i.e., has length $0$ with respect to) this form
if and only if $a^2+b^2=c^2+d^2$.  We identify $\R^4\setminus\{0\}$ with $\R\mathbb P^3$, 
by regarding each point $(a,b,c,d)\in\R^4\setminus\{0\}$ as the homogeneous coordinates
of a point in $\R\mathbb P^3$. Then, for $d\ne 0$, the corresponding affine point
$(x,y,z)=(\frac ad,\frac bd,\frac cd)$ satisfies $x^2+y^2-z^2=1$.  
In other words, the only admissible points must lie on this hyperboloid. In this case
the resulting incidence problem involves points and lines on a hyperboloid. Since each 
point can be incident to at most two lines of this kind, we get a trivial linear bound,
and an ``uninteresting'' incidence question.

Things brighten up when we pass to the symplectic group $Sp(4)$. Here flags correspond 
to sequences of nested subspaces of $\R^4$, so that some fixed symplectic form (i.e., a 
non-degenerate skew-symmetric form) vanishes on all the proper subspaces.  Again, this 
can be shown to be possible only for one- and two-dimensional subspaces, and the same
identification of $\R^4\setminus\{0\}$ with $\R\mathbb P^3$ will result in point-line pairs in $\R^3$.  
By the definition of a symplectic form, every vector has length $0$ with respect to the form.
Therefore the form annihilates all one-dimensional subspaces, and so all points in $\R^3$ 
are admissible. If we take our form to have the matrix
$$
T = \begin{bmatrix}0&1&0&0\\-1&0&0&0\\0&0&0&1\\0&0&-1&0\end{bmatrix},
$$
then it evaluates to zero on the affine span $\ell$ of a pair of points
$(a_1,b_1,c_1,d_1)$ and $(a_2,b_2,c_2,d_2)$, if and only if 
$$
(a_1, b_1, c_1, d_1) T \left( 
\begin{array}{c}
a_2 \\ b_2 \\ c_2 \\ d_2
\end{array}
\right) = 0 ,
$$
namely, $a_1b_2+c_1d_2=a_2b_1+c_2d_1$.  
Again, identifying $\R^4\setminus\{0\}$ with $\R\mathbb P^3$, we get $x_1y_2+z_1=x_2y_1+z_2$.  
This means that two points $(x_1,y_1,z_1)$ and $(x_2,y_2,z_2) = (x_1,y_1,z_1)+(u,v,w)$ 
lie on an admissible line if they satisfy the above condition. That is,
\begin{align*}
x_1(y_1+v) + z_1 & = (x_1+u)y_1 + z_1+w ,\quad\text{or} \\
y_1u-x_1v+w & = (u,v,w)\cdot (y_1,-x_1,1) = 0
\end{align*}
at every point $(x_1,y_1,z_1)\in\ell$.
In other words, for any point $(x_1,y_1,z_1)$, the admissible lines through that point 
lie on a plane with normal vector $(y_1,-x_1,1)$. These are exactly the Legendrian lines
whose incidences have been studied in Theorem~\ref{Legendrian}.

\subsection{Conclusion and Future Research}\label{Conclusion}
In this paper, we have studied several new problems in incidence geometry. 
These include a tight bound for the number of complete and partial flags in $\R^d$ and $\C^d$, 
as well as an almost tight bound on a three-dimensional variant of this problem, and two
bounds, resembling those obtained by Guth and Katz~\cite{GK2}, for two special point-line
incidence problems in three dimensions. One of these problems (involving Legendrian lines)
has been motivated by the theory of generalized flag varieties in the theory of Lie groups.

There are several interesting open problems that our work raises. The bounds obtained
in Theorems~\ref{Light} and~\ref{Legendrian} are not known to be tight, and it would
be interesting to obtain lower bounds for these problems (for the case of light-like lines,
the goal would be to improve the lower bound given in~\cite{SW}).

It would also be interesting to find additional (natural) classes of generalized 
flag varieties that lead to new incidence questions in three and higher dimensions.

\section*{Acknowledgements}
Saarik Kalia and Ben Yang would like to thank Larry Guth for suggesting this project and providing tools which were 
crucial in the discovery of our results.  They would also like to thank David Jerison and Pavel Etingof 
for repeatedly offering insightful input into our project and providing future directions of research.  
Finally, they would like to thank the SPUR program for allowing them the opportunity to conduct this research.

Work on this paper by Noam Solomon and Micha Sharir was
supported by Grant 892/13 from the Israel Science Foundation.
Work by Micha Sharir was also supported
by Grant 2012/229 from the U.S.--Israel Binational Science Foundation,
by the Israeli Centers of Research Excellence (I-CORE)
program (Center No.~4/11), and
by the Hermann Minkowski-MINERVA Center for Geometry
at Tel Aviv University.



\end{document}